\documentclass{article}
\usepackage{amsmath,amsthm,amsfonts,amssymb,amscd, amsxtra}
\usepackage{color}
\usepackage{lscape}
\usepackage{lineno,hyperref}
\modulolinenumbers[5]

\newcommand{\dist}{\text{dist}}

\newcommand{\trace}[1]{\text{tr}(#1)}
\newcommand{\Sn}{\mathbb{S}^n}
\newcommand{\Snm}{\mathcal{S}^n_{+}}
\newcommand{\Rn}{\mathbb{R}^n}
\newcommand{\Rm}{\mathbb{R}^m}
\newtheorem{theorem}{Theorem}
\newtheorem{lemma}[theorem]{Lemma}
\newtheorem{definition}{Definition}
\newtheorem{corollary}[theorem]{Corollary}
\newtheorem{proposition}[theorem]{Proposition}

\newtheorem{remark}{Remark}

\begin{document}

\title{Levenberg-Marquardt methods with inexact projections for constrained nonlinear systems}
\author{
    Douglas S. Gon\c calves
    \thanks{Departamento de Matem\'atica, Universidade Federal de Santa Catarina, Florian\'opolis, SC 88040-900, Brazil. (E-mail: {\tt
       douglas@mtm.ufsc.br}. The work of this author was supported in part by CNPq Grant 421386/2016-9.}
  \and  Max L. N. Gon\c calves
  \thanks{IME, Universidade Federal de Goi\'as, Goi\^ania, GO 74001-970, Brazil. (E-mails: {\tt
       maxlng@ufg.br} and {\tt fabriciaro@gmail.com}). The work of these authors was
    supported in part by CAPES, FAPEG/CNPq/PRONEM-201710267000532, and CNPq Grants   302666/2017-6 and 408123/2018-4.}
   \and Fabr\'icia R. Oliveira\footnotemark[2]
}
 \maketitle




\begin{abstract}
In this paper, we first propose a new 
Levenberg-Marquardt method for solving constrained (and not necessarily square) nonlinear systems.
Basically, the method combines the unconstrained Levenberg-Marquardt method with a type of feasible  inexact projection.
The local convergence of the new method as well as results on its rate are established by using an error bound condition, 
which  is weaker than the standard full-rank assumption. 
We further present and analyze a global version of the first method by means of a nonmonotone line search technique.
Finally, numerical experiments illustrating the practical advantages  of the proposed schemes are reported.
\end{abstract}



\section{Introduction}\label{sec:int3}
In this paper, we consider the following problem
\begin{equation}\label{eq:prob}
\begin{aligned}
\text{Find } x \in \mathbb{R}^n \; \text{:} & \; F(x)=0, \quad x\in C,
\end{aligned}
\end{equation}
where  $C$ is a nonempty closed convex set contained in an open set $\Omega \subset \mathbb{R}^n$ and $F: \Omega \rightarrow \mathbb{R}^m$ is a continuously differentiable function.  Throughout this paper, we will assume that
 the solution set of \eqref{eq:prob}, denoted by $C^*$, is nonempty. 
 
Problem~\eqref{eq:prob} has been the object of intense research in the last decades since many applications that arise in different areas such as engineering, chemistry, economy among others can be modeled by a constrained system of nonlinear equations. Consequently, many efficient algorithms such as trust region, interior point, active-set, Newton-type, Gauss-Newton, Levenberg-Marquardt methods have been used to solve \eqref{eq:prob}; see, for instance \cite{Behling2017,Behling2014,morini1,bellavia2006,echebest2012,Fan2013,CondG,Oliveira2017,Kanzow2004,sandra,cruz2014,MACCONI2009859,mariniquasi,martinez1997,morini2016,Porcelli,wang2016,Zhang1,Zhu2005343}.

The unconstrained pure/local Levenberg-Marquardt \cite{Levenberg1944,Marquardt1963} method recursively computes  a sequence $\{x_k\}$ as follows. 
Given $x_{k} \in \mathbb{R}^n$ and $\mu_{k}>0$,
\begin{equation}\label{eq:lmi}
x_{k+1} = x_{k} + d_{k}^U,
\end{equation}
where the step $d_{k}^U$ is the solution of the problem
\begin{equation}\label{unconstrained}
\min_{d \in \mathbb{R}^n}  \; \| F(x_{k}) + F'(x_{k}) d \|^2 + \mu_{k} \|d\|^2,
\end{equation}
or, equivalently, $d_{k}^U = - ( F'(x_{k})^T F'(x_{k}) + \mu_{k} I )^{-1}F'(x_{k})^T F(x_{k})$. 
In this case, $F'(x)$ denotes the Jacobian matrix of $F$ at $x$. 
The regularization parameter $\mu_{k} > 0$, which is not present in the Gauss-Newton method, turns the problem \eqref{unconstrained} into a strongly convex one and hence it possesses a unique solution. A classical choice of the regularization parameter is  $\mu_k = \| F(x_k) \|^2,$ for every $k\geq 0$; however,  suggestions of different regularization parameters  have been discussed, for example, in \cite{Fan2013,Zhang2003}.  We  also refer the reader to \cite{Fan2005,Yamashita2000} where convergence results of the unconstrained Levenberg-Marquardt method  and its variants have been studied.

In order to solve  constrained problems (see, e.g., \cite{Behling2017,Behling2014,Kanzow2004}), the Levenberg-Marquardt method has been adapted in two different  ways:
(i)  the constraint  $x_{k}+d\in C$ is added to the subproblem \eqref{unconstrained} (resulting in the so-called constrained Levenberg-Marquardt method);
(ii) the update  \eqref{eq:lmi} is replaced by $x_{k+1} =P_{C}( x_{k} + d_{k}^U)$, where $P_{C}$  is the  orthogonal projector  onto $C$
(arriving at the projected Levenberg-Marquadt method).
Since the subproblem in the former strategy can be relatively complicated,  depending on the feasible set $C$, 
the projected Levenberg-Marquardt method is much  more interesting mainly   when the projection steps are not  expensive.

Therefore, the  goal of this article is to present some improvements in the projected Levenberg-Marquardt method.  Since  depending on the geometry of $C$, 
the  orthogonal projection onto it neither has a closed-form nor can be easily computed,  
we first propose  a  local Levenberg-Marquardt method  in  which   inexact projections  are allowed. 
The \textit{feasible} inexact projections used in our algorithm can be easily obtained by means of an iterative method 
(e.g., the conditional gradient method \cite{jaggi13}) in the cases where computing  the exact  projections  are difficult and  expensive  (see Definition~\ref{def:IP} and Remark~\ref{remark1d} below). The local convergence of the proposed method as well as results on its rate are established by using an error bound condition, which  is weaker than the standard full-rank condition of the $F'$.  Specifically,  let $\{x_k\}$ be the sequence generated by the method and $\dist(x, C^*)$ the distance from $x$ to the solution set $C^*$.  
We show that the sequence $\{\dist(x_k, C^*)\}$ converges to zero linearly and if, additionally,  the inaccuracies of the projections tend to zero sufficiently fast, 
then the convergence is superlinear. 
Moreover, we also deduce the convergence rate for the sequence $\{ x_k \}$. 

Then, we present and analyze a globalized version of the local method. 
Basically, it  
consists of combining  our first algorithm, safeguarded by inexact projected gradient  steps, with the nonmonotone line search technique in \cite{grippo1986}.  
It is worth pointing  out that  the nonmonotone strategies have been shown more efficient than monotone ones due to the fact that enforcing the monotonicity of the function values  may make the method to converge slower. For the global method, we prove that  any accumulation point of the iterative sequence 
is a stationary point of $\min_{x\in C}   \|F(x)\|^2/2.$

It should also be pointed out that, due to the inexactnesses of the projections, 
the convergence analyses  of the proposed projected Levenberg-Marquardt schemes are, in some sense,  
more challenging.

Finally, in order to assess  the practical behavior of the new methods, some numerical experiments are reported.
In particular, we present a  scenario in which our concept of inexact projection becomes  interesting in practice.


\ \\
{\bf Outline of the paper:} Section~\ref{sec:condGmet3} introduces the concept of feasible inexact projections and  describes the local Levenberg-Marquardt method with inexact projections (LMM-IP).  The local convergence analysis of the LMM-IP is presented in Subsection~\ref{localcon}.  Section~\ref{global}  proposes and analyzes  a global version  of the LMM-IP studied in Section~\ref{sec:condGmet3}. Some preliminary numerical experiments for the proposed schemes are reported in Section~\ref{numerical}.
Finally, some concluding remarks are given in Section~\ref{remarks}.
%
\\[2mm]
{\bf Notation:}  We denote by $F'(x)$ the Jacobian matrix of $F$ at $x \in \Omega$. The inner product and its associated Euclidean norm in $\mathbb{R}^n$ are denoted by $\langle\cdot,\cdot\rangle$ and $\| \cdot \|$, respectively. The closed ball centered at $x$ with radius $r$ is denoted by $B(x,r) := \{ y \in \mathbb{R}^n \ : \ \| y - x\| \leq r \}$. 
We define by
\begin{equation}\label{dist}
\dist(x,C^*):= \inf_{y \in C^*} \| y - x \|,
\end{equation}
the distance from $x$ to the solution set $C^*$. We also represent by $\bar{x}$ a point in $C^*$ which realizes such distance, i.e.
\begin{equation}\label{dist1}
\|x - \bar{x} \| = \dist(x,C^*).
\end{equation}
For a matrix $X\in\mathbb{R}^{n\times n}$, its transpose is denoted by $X^T$, and  $ X \succeq 0$ means that
$X$ is positive semidefinite.

\section{Levenberg-Marquardt method with inexact projections}\label{sec:condGmet3}

In this section, we propose and analyze a local Levenberg-Marquardt method with inexact projections (LMM-IP) to solve \eqref{eq:prob}.  The local convergence of the proposed method as well as results on its rate are established by using an error bound condition, which  is weaker than the standard full-rank assumption. 

In order to present our algorithm, we first need to discuss a concept of approximate projection.
It is worth pointing out that, depending on the definition/geometry of $C$, computing  the orthogonal projection of a point onto $C$ can be difficult and  expensive. In order to overcome this drawback, our algorithm will  admit a certain type of  inexact projection.

\begin{definition}\label{def:IP}
Let $x\in \mathbb{R}^n$ and $\varepsilon \geq 0$ be given. We say that $ P_{C}(x,\varepsilon)$ is an  $\varepsilon$--projection of $x$ onto $C$ when
\begin{equation}\label{eq:iproj}
P_{C}(x,\varepsilon) \in C \quad \text{and} \quad \langle x - P_{C}(x,\varepsilon), y - P_{C}(x,\varepsilon) \rangle \leq \varepsilon, \quad \forall y \in C.
\end{equation}
\end{definition}

\begin{remark}\label{remark1d}
(i) Note that, if  $\varepsilon=0$, then $ P_{C}(x,0)$  corresponds to the  orthogonal projection of $x$ onto $C$, which will be denoted, simply,  by $P_{C}(x)$.  On the other hand, $P_{C}(x)$ is an  $\varepsilon$--projection of $x$ onto $C$ in the sense of Definition~\eqref{def:IP} for any $\varepsilon\geq0$.
(ii) In the case that orthogonal projection onto $C$ neither has a closed-form nor can be easily computed,
an $\varepsilon$--projection of $x$ onto $C$ can be obtained by means of an iterative method applied to solve the projection problem $\min_{y\in C}\|y-x\|^2/2$.
For example, if $C$ is bounded,  one can use
the conditional gradient (CondG) method, a.k.a. Frank-Wolfe method \cite{fw1956,jaggi13},
to obtain an inexact projection in the sense of Definition~\ref{def:IP}.
Given $z_{t}\in C$,  the $t$-th step of the CondG method first finds  $w_t$ as  a minimum of the linear function
$\langle{z_{t} - x},{\cdot - z_{t}} \rangle$ over  $ C$ and then set $z_{t+1} =(1-\alpha_t)z_{t}+\alpha_t w_t $ for some $\alpha_t \in  [0,1]$.
Hence, if $\langle{z_{t} - x}, {w_t - z_{t}} \rangle \geq -\varepsilon$ is used as a  stopping criterion in the CondG method,
we  will  have that  the output $z_{t}$ is an $\varepsilon$--projection of $x$ onto $C$.

\end{remark}

It is well-known that the projection operator $ P_{C}(\cdot)$ is nonexpansive, i.e.,
\begin{equation}\label{Pro:exp}
\| P_{C}(y) - P_{C}(x)\| \le \| x - y \|, \quad \forall x,y \in \mathbb{R}^n.
\end{equation}
In the next proposition, we establish  a similar  property for the operator $P_{C}(\cdot,\cdot)$.


\begin{proposition}\label{prop:errproj}
For any  $x,y \in \mathbb{R}^n$ and $\varepsilon\geq 0$, we have
\[
\|P_C(x,\varepsilon) - P_C(y)\| \leq \|x-y\| + \sqrt{\varepsilon}.
\]
\end{proposition}
\begin{proof}
It follows from the characterization of the orthogonal projection and Definition~\ref{def:IP}  that
$$ \langle x - P_{C}(x), P_{C}(x,\varepsilon) - P_{C}(x)\rangle \leq 0, \qquad \langle x - P_{C}(x,\varepsilon), P_{C}(x) - P_{C}(x,\varepsilon)\rangle \leq \varepsilon.$$
Combining the last two inequalities, we obtain
$$\langle x - P_{C}(x) - x + P_{C}(x,\varepsilon), P_{C}(x,\varepsilon) - P_{C}(x) \rangle \leq \varepsilon,$$
or, equivalently,
 \[
 \|P_{C}(x) - P_{C}(x,\varepsilon) \| \le \sqrt{\varepsilon}.
\]
Therefore, using the triangle inequality and  \eqref{Pro:exp}, we have
\[
\|P_C(x,\varepsilon) - P_C(y)\|\leq \|P_C(x,\varepsilon) - P_C(x)\|+\|P_C(x) - P_C(y)\|\leq \sqrt{\varepsilon}+\|x-y\| ,
\]
concluding the proof.
\end{proof}

We are now ready to formally described the  Levenberg-Marquardt method with inexact projections.
\\[3mm]
\fbox{
\begin{minipage}[h]{0.96\textwidth}
{\bf  LMM-IP}
\begin{description}
\item[ Step 0.] Let $x_0\in C$ and $\{\theta_j\}\subset[0,\infty)$  be given. Set $k=0$.
\item[ Step 1.] If $F(x_k)=0$, then {\bf stop}; otherwise, set $\mu_k := \|F(x_k)\|^2$ and compute $d_k^U \in \mathbb{R}^n$ such that
\begin{equation}\label{A1:s1}
(F'(x_k)^T F'(x_k) + \mu_k I )d_k^U= -F'(x_k)^T F(x_k).
\end{equation}
\item[ Step 2.] Define $\varepsilon_k := \theta_k^2 \| d_k^U \|^2$. Compute $P_{C}(x_k + d_k^U, \varepsilon_k)$, an $\varepsilon_k$--projection of $x_k + d_k^U$ onto $C$, and set
\begin{equation}\label{eq:lm3}
x_{k+1} := P_{C}(x_k + d_k^U, \varepsilon_k).
\end{equation}
\item[ Step 3.]  Set $k\gets k+1$, and go to Step~1.
\end{description}
\noindent
\end{minipage}
}
\\

\begin{remark} (i) Since $\mu_k>0$, it follows that the matrix of the linear system \eqref{A1:s1} is symmetric  positive definite  and hence Step~1 is  always well-defined. As a consequence, the  LMM-IP is also always well-defined.
(ii) From Step~2 of the LMM-IP and Definition~\ref{def:IP}, we have  $x_{k+1}$ satisfies, for every $k\geq0$,
\begin{equation}\label{eq:condpro}
x_{k+1}\in C, \qquad \langle x_k + d_k^U - x_{k+1}, y - x_{k+1}\rangle \le  \theta_k^2 \| d_k^U \|^2, \quad \forall y \in C.
\end{equation}
See Remark~\ref{remark1d} for some comments about our concept of the  inexact projection and how to compute it.
\end{remark}

\subsection{Local convergence of the LMM-IP}\label{localcon}

In order to analyze the local convergence of the LMM-IP, the following assumptions are made throughout this subsection.

\begin{itemize}
\item[{\bf(A0)}]  Assume $C^* \ne \emptyset$ and let $x_* \in C^*$ be an arbitrary element of the solution set.
\item[{\bf(A1)}]  There exist $L,\delta_1>0$  such that
the $F'$ is $L$-Lipschitz continuous in  $B(x_*, \delta_1)$, i.e.,
$$
\| F'(x) - F'(y) \| \leq L \| x - y \|, \quad \forall x,y \in  B(x_*, \delta_1).
$$
\item[{\bf(A2)}]  There exist $\omega,\delta_2 > 0$ such that
$\| F(x) \|$ provides a local error bound on $B(x_*,\delta_2)$, i.e.,
\begin{equation}\label{cond:errb}
\omega \, \dist(x, C^*) \leq \| F(x) \|, \quad \forall x \in B(x_*,\delta_2).
\end{equation}
\end{itemize}

It is worth mentioning that {\bf (A2)} was used in \cite{Kanzow2004} to analyze the local convergence of constrained and projected Levenberg-Marquardt methods. 
Although {\bf (A2)} may not be satisfied for zeroes of $F$ in the boundary of $C$, this condition is still weaker than the standard
full-rank condition of the Jacobian matrix (see, e.g., \cite{Kanzow2004} for more details).
We also refer the interested reader to \cite{Behling2017} for some discussions on error bound conditions in the context of projected Levenberg-Marquardt methods.

An immediate consequence of {\bf (A1)} and smoothness properties of the mapping $F$, whose proof will be  omitted,  is given  in the following proposition.


\begin{proposition}\label{prop:lip}
If {\bf (A1)} holds, then
\begin{equation}\label{eq:IPLM14}
\| F(y) - F(x) - F'(x)(y - x)\| \leq \dfrac{L}{2} \|x - y \|^2, \quad \forall x,y \in B(x_*, \delta_1).
\end{equation}
Moreover, there exists $L_0>0$ such that
\begin{equation*}\label{eq:IPLM15}
\| F(x) - F(y) \| \leq L_0 \| x - y \|, \quad \forall x,y \in B(x_*, \delta_1).
\end{equation*}
\end{proposition}

The next result summarizes some well-known properties of the unconstrained Levenberg-Marquardt method which are consequences of assumptions {\bf (A1)} and {\bf (A2)}.
\begin{proposition}\label{prop:boundsU}
Suppose that {\bf (A1)} and {\bf (A2)} hold.
If $x_k \in B(x_*,\delta/2)$, where $\delta = \min\{\delta_1,\delta_2\}$, then
\begin{equation}\label{des:dku}
\| d_k^U \| \leq c_1\dist(x_k,C^*),
\end{equation}
and
\begin{equation}\label{des:Fku}
\| F(x_k) + F'(x_k) d_k^U \| \le c_2 \dist(x_k,C^*)^2,
\end{equation}
with
\begin{equation}\label{eq:fg56}
c_1 := \sqrt{{L^2}/{(4 \omega^2)} + 1 }, \qquad c_2 := \sqrt{{L^2}/{4} + {L_0}^2}.
\end{equation}
\end{proposition}
\begin{proof}
The proof 
can be found in \cite[Lemma~2.1]{Yamashita2000}.
\end{proof}

In the following, we establish some auxiliary results which will be used to prove that the sequences $\{\dist(x_k,C^*)\}$ and $\{x_k\}$ converge.

\begin{lemma}\label{lemma:dist2}
If  $x_k, x_k+d_k^U \in B(x_*,\delta/2)$, where  $\delta = \min\{\delta_1,\delta_2\}$,  then
\begin{equation}\label{eq:taxa}
\dist(x_{k+1},C^*) \leq \theta_k  c_1 \dist(x_k,C^*)+ \dfrac{(2c_2+Lc_1^2)}{2\omega} \dist(x_k,C^*)^2,
\end{equation}
where  $c_1$ and $c_2$ are as in \eqref{eq:fg56}.
\end{lemma}
\begin{proof}
It follows from \eqref{dist} and \eqref{eq:lm3} that
\begin{equation}\label{eq:IPLM6}
\dist(x_{k+1},C^*) =  \dist( P_{C}(x_k + d_k^U,\varepsilon_k), C^* ) =  \inf_{x \in C^*} \|  P_{C}(x_k + d_k^U,\varepsilon_k) - x \|.
\end{equation}
Since $P_{C}(x)=x$ for every  $x\in C$,  we obtain, from Proposition~\ref{prop:errproj}, that
\begin{align}\label{eq:IPLM7}
\inf_{x \in C^*} \| P_{C}(x_k + d_k^U, \varepsilon_k) - x \| & =   \inf_{x \in C^*}\| P_{C}(x_k + d_k^U, \varepsilon_k) - P_C(x)\|  \nonumber\\
& \leq  \sqrt{\varepsilon_k}  + \inf_{x \in C^*} \| x_k + d_k^U-x \| \nonumber\\
& =  \sqrt{\varepsilon_k} + \dist(x_k + d_k^U, C^*).
\end{align}
 Hence, from \eqref{cond:errb}, \eqref{eq:IPLM6} and \eqref{eq:IPLM7} and the fact that $\varepsilon_k = \theta_k^2 \|d_k^U\|^2$,  we have
\begin{align}\label{eq:IPLM8}
\dist(x_{k+1},C^*) \leq  \sqrt{\varepsilon_k} + \dist(x_k + d_k^U, C^*)
 \leq  \theta_k \| d_k^U \| + \dfrac{1}{\omega} \| F(x_k + d_k^U) \|.
\end{align}
On the other hand, it follows from  \eqref{eq:IPLM14} that
\begin{equation*}
\|F(x_k + d_k^U)\| - \|F(x_k) + F'(x_k)d_k^U\| \leq  \| F(x_k) - F(x_k + d_k^U) + F'(x_k)d_k^U\| \leq \dfrac{L}{2} \|d_k^U\|^2,
\end{equation*}
which, combined with \eqref{eq:IPLM8}, \eqref{des:dku} and  \eqref{des:Fku}, yields
\begin{align*}
\dist(x_{k+1},C^*) &\leq \theta_k \| d_k^U \| + \dfrac{1}{\omega} \left(\|F(x_k) + F'(x_k)d_k^U\| + \dfrac{L}{2} \|d_k^U\|^2 \right)\\
&\leq \theta_k  c_1 \dist(x_k,C^*) + \dfrac{1}{\omega} \left( c_2 \dist(x_k,C^*)^2 + \dfrac{L}{2} c_1^2 \dist(x_k,C^*)^2 \right).
\end{align*}
Therefore, the desired inequality follows trivially from the above one.
\end{proof}

As a consequence of Lemma~\ref{lemma:dist2}, we obtain a useful corollary. 
It  shows that $x_{k+1}$ is closer to $C^*$ than $x_k$ as long as $x_k$ and $x_k + d_k^U$ are in a suitable neighborhood of $x_*$.

\begin{corollary}\label{lemma:dist}
Assume that
 $\theta_k\leq  \bar\theta  < 1/c_1$ and  $x_k, x_{k}+d_k^U\in B(x_*, \sigma/2)$, where
\begin{equation}\label{eq:123}
\sigma < \min\left\{\delta_1,\delta_2, \dfrac{4 \omega(1 - \bar \theta c_1)}{2 c_2 + L c_1^2}\right\}.
\end{equation}
Then,
\begin{equation}\label{ee:456}
\dist(x_{k+1},C^*) \leq  \eta \dist(x_k,C^*),
\end{equation}
with  $\eta:= [\bar\theta c_1  + {(2c_2+Lc_1^2)\sigma}/({4\omega})]\in(0,1)$.
\end{corollary}
\begin{proof}
First, the  inequality in \eqref{ee:456} follows   from Lemma~\ref{lemma:dist2} and the facts that $ \theta_k\leq \bar \theta$ and $x_k \in B(x_*, \sigma/2)$. Now, in view of the definition $\sigma$ in \eqref{eq:123}, we trivially have $\eta\in(0,1)$.
\end{proof}

In the next lemma, we will prove that for an initial point sufficiently close to the solution set,
the sequences $\{x_k\}$ and $\{x_k + d_k^U\}$ are contained in $B(x_*, \sigma/2)$.

\begin{lemma}\label{le:cap}
Assume that $\theta_k\leq  \bar\theta  < 1/c_1$, for every $k\geq0$, and define
\begin{equation}\label{def:r}
r :=  \dfrac{\sigma(1 - \eta)}{2(1+c_1)[1 - \eta +(1+ \bar \theta)c_1]},
\end{equation}
where $c_1$,  $\sigma$ and $\eta$ are as in Proposition~\ref{prop:boundsU}  and  Corollary~\ref{lemma:dist}.
If $x_0 \in B(x_*, r) \cap C$, then $x_k, x_{k}+d_k^U\in B(x_*, \sigma/2)$ for every $k\geq0$.
\end{lemma}
\begin{proof}
We will proceed by induction on $k$.  Since  $r < \sigma/2 < \delta/2$, where \linebreak$\delta = \min\{\delta_1,\delta_2\}$, we have
\[ [B(x_*, r)\cap C]\subset [B(x_*, \sigma/2)\cap C]\subset [B(x_*, \delta/2)\cap C]. \]
Therefore,  $x_0 \in B(x_*, \sigma/2)\cap C$. Moreover, using \eqref{des:dku}, we obtain
\begin{equation*}
 \|x_0+d_0^U- x_*\|\leq  \|x_0- x_*\|+\|d_0^U\|\leq r+  c_1 \dist(x_0,C^*)\leq r+  c_1 \|x_0 - x_*\| \leq   (1+c_1) r.
\end{equation*}
Since, in particular,  $r< \sigma/[2(1+c_1)]$, we conclude that $x_0+d_0^U \in B(x_*,\sigma/2)$. Now, suppose that $x_l,  x_{l}+d_l^U \in B(x_*, \sigma/2)$ for all $l = 0, \ldots, k$ and let us show that $x_{k+1},  x_{k+1}+d_{k+1}^U\in B(x_*, \sigma/2)$. Using \eqref{eq:lm3} and Proposition~\ref{prop:errproj}, we find that
\begin{equation*}
\|x_{k+1} - x_*\| = \|P_{C}(x_k + d_k^U, \varepsilon_k) - P_{C}(x_*)\| \leq \|x_k + d_k^U - x_*\| + \sqrt{\varepsilon_k}.
\end{equation*}
By the  triangle inequality and the facts that $\varepsilon_k = \theta_k^2 \|d_k^U\|^2$ and $\theta_k\leq  \bar\theta $, we obtain
\begin{equation*}
\|x_{k+1} - x_*\|   \leq \|x_k -  x_*\| + (1 + \bar \theta)\|d_k^U\|.
\end{equation*}
Hence,
\begin{equation}\label{eq:IPLM13}
\|x_{k+1} - x_*\| \leq \|x_0 -  x_*\| + (1 + \bar \theta) \sum_{l = 0}^k\|d_l^U\| \leq r + (1 +\bar \theta) c_1 \sum_{l = 0}^k \dist(x_l,C^*),
\end{equation}
where the last inequality follows from the facts that $x_l \in B(x_*, \sigma/2)\subset B(x_*, \delta/2)$, for all $l = 0, \ldots, k$, 
$x_0 \in B(x_*, r)$ and \eqref{des:dku}.
On the other hand, since $x_l,  x_{l}+d_l^U \in B(x_*, \sigma/2)$ for all $l = 0, \ldots, k$, by Corollary~\ref{lemma:dist}, we have
\begin{equation}\label{eq:IPLM1}
\dist(x_l,C^*) \leq \eta \dist(x_{l-1},C^*) \leq \eta^2 \dist(x_{l-2},C^*) \leq \ldots \leq \eta^l \dist(x_{0},C^*) \leq \eta^l r,
\end{equation}
for all $l = 0, \ldots, k$, where the last inequality is due to the fact that $x_0 \in B(x_*, r)\cap C$.
Hence, it follows from \eqref{eq:IPLM13} and \eqref{eq:IPLM1} that
\begin{equation*}
\|x_{k+1} - x_*\| \leq r + (1 + \bar \theta) c_1 r \sum_{l = 0}^{\infty} \eta^l.
\end{equation*}
Since $\eta \in (0,1)$, we have $
\sum_{l = 0}^{\infty} \eta^l = {1}/{(1 - \eta)}.
$
Hence, using the last inequality, we obtain
\begin{align}\label{eq:IPLM16}
\|x_{k+1} - x_*\| \leq r +  \frac{(1 + \bar \theta) c_1 r}{1 - \eta}.
\end{align}
As, in particular,  $r < (1 - \eta)\sigma/[2(1-\eta + (1+\bar \theta)c_1)]$, we conclude that $x_{k+1}\in B(x_*, \sigma/2)$.
It remains to  prove that $x_{k+1}+d_{k+1}^U\in B(x_*, \sigma/2)$.  Since  $x_{k+1}\in B(x_*, \sigma/2)$, it follows from \eqref{des:dku} that
\begin{equation*}\label{eq:IPLM17}
\|x_{k+1} + d_{k+1}^U - x_*\| \leq \|x_{k+1} - x_*\| + \|d_{k+1}^U\| \leq (1 + c_1) \|x_{k+1} - x_*\|,
\end{equation*}
which, combined with   \eqref{eq:IPLM16} and the definition of $r$ in  \eqref{def:r}, yields
\begin{equation*}\label{eq:IPLM18}
\|x_{k+1} + d_{k+1}^U - x_*\| \leq (1 + c_1) \left[r +  \frac{(1 + \bar \theta) c_1 r}{1 - \eta}\right] = \dfrac{\sigma}{2},
\end{equation*}
i.e., $x_{k+1}+d_{k+1}^U\in B(x_*, \sigma/2)$ and then the proof is complete.
\end{proof}

We are now ready to prove the convergence of the sequences $\{\dist(x_k, C^*)\}$ and $\{x_k\}$.

\begin{theorem}\label{conv:dist}
Assume that $\theta_k\leq  \bar\theta  < 1/c_1$, for every $k\geq0$.
Let $\{x_k\}$ be the sequence generated by the LMM-IP with starting point $x_0 \in B(x_*, r) \cap C$, where $r$ is as in \eqref{def:r}.
Then,
\begin{itemize}
\item[(a)] the sequence $\{\dist(x_k, C^*)\}$ converges to zero linearly. If, additionally,  $\lim_{k\to \infty}\theta_k=0$, the convergence is superlinear;
\item[(b)] the sequence $\{x_k\}$ converges to a point belonging to $ C^*$.
\end{itemize}
\end{theorem}
\begin{proof}
(a) The first part follows immediately from Corollary~\ref{lemma:dist} and Lemma~\ref{le:cap}.  Now, the second one  follows by dividing \eqref{eq:taxa} by $\dist(x_k,C^*)$
and taking the limit as $k\to \infty$.
\\[2mm]
(b) Since $\{\dist(x_k, C^*)\}$ converges to zero and $\{x_k\} \subset B(x_*,\sigma/2)\cap C $, it suffices to show that $\{x_k\}$ converges.
Let us prove that $\{x_k\}$ is a Cauchy sequence. To this end, take $p,q \in \mathbb{N}$ with $p \geq q$. It follows from Proposition~\ref{prop:errproj} and the facts that $\varepsilon_k = \theta_k^2 \|d_k^U\|^2$ and   $\{x_k\} \subset  C $ that
\begin{align*}
\|x_p - x_q\| &= \|P_{C}(x_{p-1} + d_{p-1}^U, \varepsilon_{p-1}) - P_{C}(x_q)\| \\
&\leq \|x_{p-1} + d_{p-1}^U - x_q\| +  \theta_{p-1} \|d_{p-1}^U\|\\
&\leq \|x_{p-1} - x_q\| + (1+ \theta_{p-1}) \|d_{p-1}^U\|,
\end{align*}
 Repeating the process above, we get
\begin{align*}
\|x_p - x_q\| &\leq  (1+ \theta_{q}) \|d_{q}^U\| + \ldots  +(1+ \theta_{p-2}) \|d_{p-2}^U\| +  (1+\theta_{p-1}) \|d_{p-1}^U\|,
\end{align*}
which, combined with the fact $\theta_k\leq  \bar\theta$, for every $k\geq0$, yields
\begin{align*}
\|x_p - x_q\| &\leq  (1+ \bar\theta)  \sum_{l = q}^{p-1}\|d_{l}^U\| \leq (1+\bar\theta) \sum_{l = q}^{\infty}\|d_{l}^U\|.
\end{align*}
Now, by \eqref{des:dku} and \eqref{eq:IPLM1}, we have
\begin{align*}
\|d_{l}^U\| &\leq  c_1 \dist(x_l,C^*) \leq c_1 \eta^l r.
\end{align*}
Combining the last two inequalities, we obtain
\begin{align*}
\|x_p - x_q\| & \leq (1+ \bar\theta)c_1  r \sum_{l = q}^{\infty}\eta^l=(1+ \bar\theta)c_1  r\left( \sum_{l = 0}^{\infty}\eta^l- \sum_{l = 0}^{q-1}\eta^l \right).
\end{align*}
As $\eta\in (0,1)$, taking the limit in the last inequality as $q\rightarrow\infty$, we obtain $\|x_p - x_q\|\to 0$. Therefore,
 $\{x_k\}$ is a Cauchy sequence and hence it converges. Let $\bar{x} = \lim_{k \rightarrow \infty} x_k$.
Since $x_k \in C, \forall k$, and $C$ is closed, then $\bar{x} \in C$.
Moreover, because $\omega \dist(x_k, C^*) \le \| F(x_k) \| \le L_0 \dist(x_k, C^*)$ and $\dist(x_k, C^*) \rightarrow 0$
as $k \rightarrow \infty$, we conclude that $\bar{x} \in C^*$.
\end{proof}

Before analyzing the convergence rates of the sequence $\{x_k\}$, let us first establish the following result.

\begin{lemma}\label{lemma:dku}
Assume that $\theta_k\leq  \bar\theta $ for every $k\geq0$ with
\begin{equation}\label{bar_theta}
\bar \theta <  \dfrac{-(1+ 4 c_1) + \sqrt{(1 + 4 c_1)^2 + 8}}{8c_1},
\end{equation}
where $c_1,c_2$ are as in \eqref{eq:fg56}. Let $r$ be as in \eqref{def:r} and
$\{x_k\}$ be the sequence generated by the LMM-IP with starting point $x_0 \in B(x_*, r) \cap C$ converging to its limit point $\bar x$. Then, for all $k \in \mathbb{N}$ sufficiently large, there exist positive constants $c_3$, $c_4$ and $c_5$ such that
\begin{itemize}
\item[(a)] $\dist(x_k,C^*)\leq c_3\|d_k^U\|$;
\item[(b)] $\|d_{k+1}^U\|\leq \theta_k  c_1^2 c_3\|d_{k}^U\| +  c_4\|d_{k}^U\|^2      \leq \bar \theta  c_1^2 c_3\|d_{k}^U\| +  c_4\|d_{k}^U\|^2$;
\item[(c)] $c_5\|x_k - \bar x\| \leq \|d_k^U\| \leq c_1 \|x_k - \bar x\|.$
\end{itemize}
\end{lemma}
\begin{proof}
(a) Since $\{x_k\}\subset C$,  we obtain, from  Proposition~\ref{prop:errproj}, that
\begin{equation}\label{eq:LMMIP16}
\|d_k^U\| = \|x_k + d_k^U - x_k\| \geq \|P_C(x_k + d_k^U,\varepsilon_k) - P_C(x_k)\| - \sqrt{\varepsilon_k}.
\end{equation}
Using \eqref{eq:lm3} and the facts that $\varepsilon_k = \theta_k^2 \|d_k^U\|^2$ and $\theta_k\leq  \bar\theta$, we conclude that
\begin{equation*}
\|d_k^U\| \geq \|x_{k+1} - x_k\| - \bar \theta \|d_k^U\|,
\end{equation*}
which implies
\begin{equation*}
(1 + \bar \theta)\|d_k^U\| \geq \|x_{k+1} - x_k\|.
\end{equation*}
Now let $\bar x_{k+1} \in C^*$ satisfying $\dist(x_{k+1},C^*) = \|x_{k+1} - \bar x_{k+1}\|$. Hence, from the previous inequality, we have
\begin{align}\label{eq:LMMIP17}
(1+\bar\theta)\|d_k^U\| &\geq \|\bar x_{k+1} - x_k\| - \|x_{k+1} - \bar x_{k+1}\| \nonumber \\
&\geq \dist(x_k,C^*) - \dist(x_{k+1},C^*) \nonumber \\
&\geq \left[1 - \bar\theta c_1  - \frac{(2c_2 + L c_1^2)}{2\omega}\dist(x_k,C^*)\right] \dist(x_k,C^*),
\end{align}
where the last inequality follows from the Lemma~\ref{lemma:dist2} and fact $\theta_k\leq  \bar\theta$. Since $\{\dist(x_k, C^*)\}$ converges to zero (see,  Theorem~\ref{conv:dist}(a)) we may assume, without loss of generality, that
\begin{equation}\label{LMM-19}
\frac{(2c_2 + L c_1^2)}{2\omega}\dist(x_k,C^*) < \frac{1}{2},
\end{equation}
for all $k \in \mathbb{N}$ sufficiently large. Hence, combining \eqref{eq:LMMIP17} and \eqref{LMM-19}, we have
\begin{align*}
(1+\bar\theta)\|d_k^U\| & \geq \left(\frac{1}{2} - \bar\theta c_1\right) \dist(x_k,C^*),
\end{align*}
which, combined with the fact that  \eqref{bar_theta} implies that $\bar \theta<1/(2c_1)$, proves item(a) with $c_3:= (1+\bar \theta)/(1/2 - \bar\theta c_1)$.
\\[2mm]
(b) It follows from \eqref{des:dku} and Lemma~\ref{lemma:dist2} that
\begin{align*}
\|d_{k+1}^U\| &\leq c_1 \dist(x_{k+1},C^*)\\
&\leq c_1 \left[\theta_k  c_1 \dist(x_k,C^*) + \dfrac{(2c_2 + L c_1^2)}{2 \omega}  \dist(x_k,C^*)^2 \right]\\
& \leq \theta_k  c_1^2 c_3 \|d_{k}^U\| + \dfrac{(2 c_1 c_2 + L c_1^3)c_3^2}{2\omega}  \|d_k^U\|^2,
\end{align*}
where the last inequality follows from item(a). Therefore, using $\theta_k\leq  \bar\theta$,   item(b) follows  with $c_4:= (2 c_1 c_2 + L c_1^3)c_3^2 /(2\omega)$.
\\[2mm]
(c) The second inequality follows easily from \eqref{des:dku}. In order to verify the first inequality, let $k\in \mathbb{N}$ sufficiently large such that item(b) applies and $c_4 \|d_k^U\| < 1/4$ holds. Moreover, it follows from \eqref{bar_theta} that  $\bar \theta  c_1^2 c_3 < 1/4$. Therefore, $\bar \theta  c_1^2 c_3+ c_4 \|d_k^U\| < 1/2$ and hence, from item~(b), we conclude that $\|d_{k+1}^U\| \leq (1/2) \|d_k^U\|$. Hence, for all $j = 0, 1, 2, \ldots,$ we obtain
\begin{equation}\label{eq:LMMIP18}
\|d_{k+j}^U\| \leq \left(\frac{1}{2}\right)^j \|d_k^U\|.
\end{equation}
On the other hand, we have
\begin{align*}
\|x_k - x_{k+l}\| &= \|P_C(x_{k}) - P_C(x_{k+l-1} + d_{k+l-1}^U,\varepsilon_{k+l-1})\|\\
&\leq \|x_k - x_{k+l-1} - d_{k+l-1}^U\| + \sqrt{\varepsilon_{k+l-1}}\\
&\leq \|x_{k} - x_{k+l-1}\| + (1 + \theta_{k+l-1}) \|d_{k+l-1}^U\|.
\end{align*}
Repeating the process above, we get
\begin{align*}
\|x_k - x_{k+l}\| &\leq  (1+ \theta_{k}) \|d_{k}^U\| + \ldots  +(1+ \theta_{k+l-2}) \|d_{k+l-2}^U\| +  (1+\theta_{k+l-1}) \|d_{k+l-1}^U\|,
\end{align*}
which, combined with the fact $\theta_k\leq  \bar\theta$, for every $k \geq 0$, and \eqref{eq:LMMIP18}, yields
\begin{align}\label{LMNIP:eq19}
\|x_k - x_{k+l}\| &\leq  (1+ \bar\theta)  \sum_{j = 0}^{l-1}\|d_{k+j}^U\| \leq (1+ \bar\theta) \|d_k^U\|\sum_{j = 0}^{l-1} \left(\frac{1}{2}\right)^j .
\end{align}
Taking the limit in \eqref{LMNIP:eq19} as $l\rightarrow\infty$, we obtain
\begin{equation*}
\|x_k - \bar x\| = \lim_{l\rightarrow \infty} \|x_k - x_{k+l}\| \leq (1+ \bar\theta) \|d_k^U\|\sum_{j = 0}^{\infty} \left(\frac{1}{2}\right)^j.
\end{equation*}
Since $\sum_{j = 0}^{\infty} \left(\frac{1}{2}\right)^j = 2$, we conclude, from inequality above, that
\begin{equation*}
\|x_k - \bar x\| \leq 2 (1+ \bar\theta) \|d_k^U\|,
\end{equation*}
which implies the item(c) with $c_5: = 1/[2(1 + \bar \theta)]$.
\end{proof}

The following theorem proves the local convergence rate of the sequence $\{x_k\}$ generated by the LMM-IP.


\begin{theorem}
There exist a positive constant $\alpha$ such that if  $ \theta_k \in [0, \alpha)$ for every $k\geq0$, then  the sequence $\{x_k\}$  converges linearly to its limit point $\bar x$.   If, additionally,  $\lim_{k\to \infty}\theta_k=0$,  the convergence is superlinear.
\end{theorem}
\begin{proof}    Let $\alpha_1$ be  such that $\alpha_1 < { \left({-(1+ 4 c_1) + \sqrt{(1 + 4 c_1)^2 + 8}}\right)/({8c_1}})$.
Hence, if  $\ \theta_k \in [0, \alpha_1)$, it follows from  items (b) and (c) of Lemma~\ref{lemma:dku} with $\bar \theta=\alpha_1$ that
\begin{equation}\label{des:key}
c_5 \|x_{k+1} - \bar x\| \leq \|d_{k+1}^U\|  \leq   \alpha_1  c_1^2 c_3  \|x_k - \bar x\| + c_1^2 c_4 \|x_k - \bar x\|^2,
\end{equation}
where $c_3:= (1+\alpha_1)/(1/2 - \alpha_1 c_1)$, $c_4:= (2 c_1 c_2 + L c_1^3)c_3^2 /(2\omega)$ and \linebreak$c_5: = 1/[2(1 + \alpha_1)]$.
Hence, dividing the last inequality by $\|x_k - \bar x\|$ and taking limit as $k \rightarrow \infty$, results in
$$
\lim_{k \rightarrow \infty} \dfrac{\|x_{k+1} - \bar x\|}{\|x_k - \bar x\|}= \dfrac{4c_1^3 \alpha_1(1 + \alpha_1)^2}{(1 - 2 \alpha_1  c_1)},
$$
which implies $\{x_k\}$ converges linearly to $\bar{x}$ as long as 
$$
{4c_1^3 \alpha_1 (1 + \alpha_1)^2}/[{(1 - 2 \alpha_1 c_1)}] < 1.
$$
Let $g:[0,+\infty) \rightarrow \mathbb{R}$ be defined by
\begin{equation*}
g(\alpha) := \dfrac{4c_1^3 \alpha (1 + \alpha)^2}{(1 - 2 \alpha c_1)} - 1.
\end{equation*}
Since $g(0) = -1 < 0$ and $g$ is a continuous function, there exists a positive constant $\alpha_2$ such that for all $\alpha\in [0,\alpha_2)$, 
we have  $g( \alpha) < 0$, and hence  ${4c_1^3 \alpha (1 + \alpha)^2}/[{(1 - 2 \alpha c_1)}] < 1$. 
Therefore, the result now  follows by taking $\alpha=\min\{\alpha_1,\alpha_2\}$.

Let us now prove the second part. Similarly to the first part, it can be proven that  if $ \theta_k \in [0, \alpha_1)$ for every $k\geq0$, 
where $$
\alpha_1 < { \left({-(1+ 4 c_1) + \sqrt{(1 + 4 c_1)^2 + 8}}\right)/({8c_1}}),
$$
then
$$
\lim_{k \rightarrow \infty} \dfrac{\|x_{k+1} - \bar x\|}{\|x_k - \bar x\|}= \dfrac{4c_1^3 (1 + \alpha_1)^2}{(1 - 2 \alpha_1  c_1)}\lim_{k \rightarrow \infty} \theta_k.
$$
As  $\lim_{k\to \infty}\theta_k=0$, the last equality implies the superlinear  convergence of the sequence $\{x_k\}$.
\end{proof}

\section{Global version of the  LMM-IP}\label{global}

In this section, our aim is to propose and analyze  a global version of the  Levenberg-Marquardt method with inexact projections  studied in the previous section.
Basically, the global method consists of combining  the  local LMM-IP method, safeguarded with inexact projected gradient steps, 
with the nonmonotone line search technique of \cite{grippo1986}, 
in order to guarantee a nonmonotone decrease of the merit function
\begin{equation}\label{pro2}
f(x) := \frac{1}{2} \|F(x)\|^2.
\end{equation}
The formal description of the global LMM-IP (G-LMM-IP) is given below. \\
\ \\
\fbox{
\begin{minipage}[h]{0.96\textwidth}
{\bf  G-LMM-IP }
\begin{description}
\item[ Step 0.] Let $x_0\in C$, an integer $M\geq0$, $\eta_1>0$, $\eta_3 > \eta_2 > 0$, $\gamma, \beta \in (0,1)$ and $\{\theta_j\}\subset[0,\infty)$ be given.
Set $k=0$ and $m_0 = 0$.
\item[ Step 1.] If $F(x_k)=0$, then {\bf stop}; otherwise, set $\mu_k := \|F(x_k)\|^2$ and compute $d_k^U \in \mathbb{R}^n$ such that
\begin{equation}\label{A1:s21}
(F'(x_k)^T F'(x_k) + \mu_k I )d_k^U= -F'(x_k)^T F(x_k).
\end{equation}
\item[ Step 2.] Define $\varepsilon_k := \theta_k^2 \|  d_k^U \|^2$ and  compute  $P_{C}(x_k + d_k^U, \varepsilon_k)$ .
Set
\begin{equation}\label{eq:lm32}
\bar d_{k} := P_{C}(x_k + d_k^U, \varepsilon_k) - x_k.
\end{equation}
If $|\langle \nabla f(x_k), \bar d_k\rangle| >  \eta_1 \|\bar d_k\|^2$ and  $ \eta_2 \|\nabla f(x_k)\|\leq \|\bar d_k\|\leq  \eta_3 \|\nabla f(x_k)\|$, 
then set $d_k = -\text{sgn}(\langle \nabla f(x_k), \bar d_k\rangle) \bar{d}_k$ 
 and go to \textbf{Step~4}.
\item[ Step 3.] Compute  $y_k\in C$ such that
\begin{equation}\label{de:576}
 \langle x_k - \nabla f(x_{k})-y_{k},x - y_{k} \rangle \leq \varepsilon_k:= \theta_k^2\|y_{k} - x_k\|^2,  \quad \forall \; x \in C,
\end{equation}
and set $d_k=y_k-x_k$.
\item[ Step 4.] Set $\alpha = 1$. Do $\alpha = \beta \alpha$, while
\begin{equation}\label{generalcondicion}
f(x_k + \alpha d_k) > \max_{0\leq j \le m_k} \{f(x_{k-j})\} + \gamma \alpha \langle \nabla f(x_k), d_k\rangle.
\end{equation}
\item[ Step. 5] Set $\alpha_k = \alpha$, update $x_{k+1} = x_k + \alpha d_k$,  $k \leftarrow k+1$ and \linebreak$m_k \leq \min\{m_{k-1}+1,M\}$, and go to \textbf{Step~1}.
\end{description}
\noindent
\end{minipage}
}

\begin{remark}
(i) Conditions on the search directions  $\bar d_k$ in  Step~2 are necessary in order to guarantee that any accumulation point of  $\{x_k\}$ is a stationary point of \eqref{pro2}.
(ii) It is easy to see that if  $y_k$ is the  orthogonal projection of $x_k - \nabla f(x_{k})$ onto $C$ (i.e., $y_k=P_{C} (x_k - \nabla f(x_{k})$),  then $y_k$ trivially satisfies \eqref{de:576}.
(iii) If $x_{k+1}=x_k$, then $d_k=0$ was necessarily  given by Step~3 of the G-LMM-IP and hence  $x_k$ is a stationary point of $\min_{x\in C} \|F(x)\|^2/2$  (i.e., $ \langle\nabla f(x_{k}),x - x_{k} \rangle \geq 0, $  for all $ x \in C$).
\end{remark}

The next theorem guarantees that the G-LMM-IP is well-defined, i.e., the Step~4 in the G-LMM-IP is satisfied in a finite number of backtrackings.  
In addition, we will also show that all limit points of the sequence generated by the G-LMM-IP are stationary points. It is worth pointing out that the proof of the next result is based on the one of \cite[Theorem~1]{grippo1986}.

\begin{theorem}
 Assume that $\Omega_0 = \{ x \in C \ : f(x) \le f(x_0) \}$ is bounded and $\theta_k \leq \bar \theta < 1$, for all $k\geq 0$. 
Then, the G-LMM-IP is well defined and any accumulation point of the sequence $\{x_k\}$ is a stationary point of $\min_{x\in C} \{f(x)\}.$
\end{theorem}
\begin{proof}
We will first prove that there exist positive constants $\tau_1$,  $\tau_2$ and $\tau_3$ such that the search direction $d_k$ satisfies
\begin{equation}\label{eq:1}
\langle \nabla f(x_k), d_k \rangle \leq -\tau_1\|d_k\|^2
\end{equation}
and
\begin{equation}\label{eq:3}
\tau_2 \|P_C(x_k-\nabla f(x_k))-x_k\|\leq \|d_k\|\leq \tau_3 \|\nabla f(x_k)\|,
\end{equation}
for every $k\geq0$. If $d_k$ is given by Step~2 of the G-LMM-IP, we trivially have that \eqref{eq:1} and  the second inequality in
\eqref{eq:3} hold with $\tau_1=\eta_1$ and  $\tau_3=\eta_3$.  On the other hand,  from  Step~2 of the G-LMM-IP, \eqref{Pro:exp} and the  fact that $x_k\in C$, we obtain
\[
 \|d_k\|\geq  \eta_2 \|\nabla f(x_k)\|\geq  \eta_2\|P_C(x_k-\nabla f(x_k))-x_k\|,
\]
which implies that the first inequality in \eqref{eq:3} holds with $\tau_2=\eta_2$.

Let us now prove that if $d_k$ is given by the Step~3 of the G-LMM-IP, then inequalities  \eqref{eq:1} and \eqref{eq:3} are also satisfied.  From \eqref{de:576} with $x = x_k$ and the fact that $d_k = y_k - x_k$, we have
$$\langle \nabla f(x_{k}), d_k \rangle \leq \theta_k^2\|d_k\|^2 - \|d_k\|^2 = (\theta_k^2 - 1) \|d_k\|^2,$$
which, combined with the fact that $\theta_k \leq \bar \theta$ for all $k\geq 0$, yields
$$\langle \nabla f(x_{k}), d_k \rangle \leq  -(1 - \bar \theta^2) \|d_k\|^2. $$
Hence, inequality \eqref{eq:1} holds with $\tau_1 = (1 - \bar \theta^2)$.
By Step~3 of the G-LMM-IP, Proposition~\ref{prop:errproj} and the  fact that $x_k\in C$, we have
\begin{align*}
\|d_k\| &= \|P_{C}(x_k - \nabla f(x_k), \varepsilon_k) - P_{C}(x_k) \|\\
& \leq \|x_k - \nabla f(x_k) - x_k\| + \sqrt{\varepsilon_k}\\
& = \|\nabla f(x_k)\| + \theta_k \|d_k\|,
\end{align*}
where the last equality follows from $\varepsilon_k = \theta_k^2 \|d_k\|^2$.
Hence, as $ \theta_k \leq \bar \theta<1 $ for all $k\geq 0$, we conclude that
$$ \|d_k\| \leq \|\nabla f(x_k)\|/(1 - \bar \theta), $$
which implies that the second inequality in \eqref{eq:3} holds with $\tau_3 = 1/(1 - \bar \theta)$.
On the other hand,  from Step~3 of the G-LMM-IP, we obtain
\begin{equation}\label{eq:4}
\|d_k\| =  \|P_C(x_k-\nabla f(x_k), \varepsilon_k)- x_k\|.
\end{equation}
Note that,
\begin{equation}\label{eq:5}
\|P_C(x_k - \nabla f(x_k))-x_k\|\leq \sqrt{\varepsilon_k} + \|P_C(x_k-\nabla f(x_k),\varepsilon_k)-x_k\|.
\end{equation}
Indeed, by the triangle inequality, we find
\begin{align*}
\|P_C(x_k - \nabla f(x_k))-x_k\|&\leq \|P_C(x_k - \nabla f(x_k))- P_C(x_k-\nabla f(x_k),\varepsilon_k)\| + \ \\
& \quad \ \|P_C(x_k-\nabla f(x_k),\varepsilon_k) - x_k\|\\
&\leq \sqrt{\varepsilon_k}  + \|P_C(x_k-\nabla f(x_k),\varepsilon_k) - x_k\|,
\end{align*}
where last inequality is due to Proposition~\ref{prop:errproj}.
Thus, combining \eqref{eq:4} and \eqref{eq:5}, we have
\begin{align}\label{eq:6}
\|d_k\| &\geq  \|P_C(x_k - \nabla f(x_k))- x_k\| - \sqrt{\varepsilon_k}\nonumber\\
& \geq \|P_C(x_k - \nabla f(x_k))- x_k\| - \bar \theta \|d_k\|,
\end{align}
where in the last inequality we also used  the facts that $\varepsilon_k = \theta_k^2 \|d_k\|^2$ and $ \theta_k \leq \bar \theta<1 $ for all $k\geq 0$. Therefore, from  \eqref{eq:6}, we obtain
\[
(1 +  \bar \theta) \|d_k\| \geq  \|P_C(x_k - \nabla f(x_k))- x_k\|,
\]
which implies that the first inequality in \eqref{eq:3} holds with $\tau_2 = 1/(1 +  \bar \theta)$.


Let us now show that any accumulation point of the $\{x_k\}$ is a stationary point of $\min_{x\in C} \{ f(x)=  \|F(x)\|^2/2\}$ 
by adapting the proof presented in \cite[Theorem~1]{grippo1986}.

Let $l(k)$ be an integer such that $k - m_k \leq l(k)\leq k$ and
\[
f(x_{l(k)}) = \max_{0 \leq j \leq m_k} f(x_{k-j}).
\]
Since $m_{k+1} \le m_k + 1$, it follows that $\{f(x_{l(k)})\}$ is monotonically nonincreasing, 
and from the boundness of $\Omega_0$, we ensure that $\{f(x_{l(k)}) \}$ has a limit. 
Then, from \eqref{generalcondicion}, for $k>M$, we have that 
\begin{align}\label{eq:teoglo}
f(x_{l(k)}) & = f(x_{l(k) - 1} + \alpha_{l(k) - 1} d_{l(k) - 1})\nonumber\\
&\leq \max_{0\leq j \leq m_{l(k)-1}} \{f(x_{l(k) - 1 - j})\} + \gamma \alpha_{l(k)-1}\langle \nabla f(x_{(l(k)-1)}),d_{(l(k)-1)}\rangle\nonumber\\
&= f(x_{l(l(k)-1)}) +  \gamma \alpha _{(l(k)-1)}\langle \nabla f(x_{(l(k)-1)}),d_{(l(k)-1)}\rangle.
\end{align}
Now, because $ \alpha _{(l(k)-1)} > 0$ and $\langle \nabla f(x_{(l(k)-1)}),d_{(l(k)-1)}\rangle < 0$, 
by taking limits in \eqref{eq:teoglo}, it follows that $\displaystyle lim_{k \rightarrow \infty} \alpha _{(l(k)-1)}\langle \nabla f(x_{(l(k)-1)}),d_{(l(k)-1)}\rangle = 0$. 
Moreover, from \eqref{eq:1} and \eqref{eq:3}, we conclude that 
$$
\displaystyle lim_{k \rightarrow \infty} \alpha _{(l(k)-1)} \| P_{C}(x _{(l(k)-1)} - \nabla f(x _{(l(k)-1)}) ) - x_{(l(k)-1)} \|^2 = 0, 
$$
and following the reasoning in the proof of \cite[Theorem~1(a)]{grippo1986}, we can write
\begin{equation}\label{eq:key}
\displaystyle lim_{k \rightarrow \infty} \alpha _{k} \| P_{C}(x _{k} - \nabla f(x _{k)}) ) - x_{k} \|^2 = 0.
\end{equation}

Now, let $\tilde x \in C$ be an accumulation point of $\{x_k\}$, and relabel $\{x_k\}$ a subsequence converging to $\tilde x$. 
By \eqref{eq:key}, either $\| P_{C}(x _{k} - \nabla f(x _{k)}) ) - x_{k} \| \rightarrow 0$, which implies by continuity that 
$\| P_{C}(\tilde x - \nabla f(\tilde x) ) - \tilde x \| = 0$, or 
%
%
%
%
%
%
%
%
there exists a subsequence $\{x_k\}_K$ such that $\displaystyle\lim_{k \in K}\alpha _k = 0$. 
In this last case, let $\alpha _k$ be chosen in the Step~4 of the G-LMM-IP such that $\alpha _k = \bar{\alpha } _k/2$, where $\bar\alpha _k$ was the last step that fail in \eqref{generalcondicion}, i.e.,
\begin{align}\label{eq:ar}
f(x_k + \bar{\alpha }_k d_k) >  \max_{0 \leq j \leq m_k} \{f(x_{k-j})\} +  \gamma \bar{\alpha }_k \langle \nabla f(x_k), d_k \rangle \geq f(x_k) +  \gamma \bar{\alpha }_k\langle \nabla f(x_k), d_k\rangle.
\end{align}
By the mean value theorem, there exists  $\zeta_k \in [0,1]$ such that \eqref{eq:ar} can be written as
\begin{align}\label{eq:arm}
\langle \nabla f(x_k + \zeta_k s_k), s_k \rangle = f(x_k + s_k) - f(x_k) > \gamma \langle \nabla f(x_k) , s_k \rangle,
\end{align}
where  $s_k := \bar{\alpha }_k d_k$. 
 Notice that $s_k$ goes to zero as $k \in K$ goes to infinity, because $\lim_{k \in K}\alpha _k = 0$ and $\|d_k\|$ is bounded. 
 So, from \eqref{eq:arm}, we have
\begin{align}\label{eq:armi}
\left \langle \nabla f(x_k + \zeta_k s_k), \frac{s_k}{\|s_k\|} \right\rangle  > \gamma \left \langle \nabla f(x_k) , \frac{s_k}{\|s_k\|} \right\rangle.
\end{align}
By taking limit in \eqref{eq:armi} as $k\in K_1$ goes to infinity, where $K_1$ is such that
$$\lim_{k\in K_1} \frac{s_k}{\|s_k\|} = s,$$ we obtain  $(1 - \gamma)\langle \nabla f(\tilde x) , s \rangle \geq 0$. Since $(1 - \gamma ) > 0$, we have
\begin{align}\label{eq:te}
 \langle  \nabla f(\tilde x), s \rangle \geq 0.
\end{align}
On the other hand, it  follows  from \eqref{eq:1} that  $\langle  \nabla f(x_k), d_k \rangle < 0$ for all $k\geq 0$, which combined with the fact that  $s_k = \bar{\alpha }_k d_k$, yields
\[
\left \langle \nabla f(x_k), \frac{s_k}{\|s_k\|} \right \rangle < 0, \quad  \forall k\geq 0.
\]
Hence, by taking limit in the last inequality, we conclude that $\langle  \nabla f(\tilde x), s \rangle \leq  0$, which combined with \eqref{eq:te}, yields $\langle  \nabla f(\tilde x), s \rangle = 0$. Using  the  definition of $s_k$, \eqref{eq:1} and \eqref{eq:3}, we have
\[
\left \langle \nabla f(x_k), \frac{s_k }{\|s_k\|} \right\rangle=\left \langle \nabla f(x_k), \frac{d_k }{\|d_k\|} \right\rangle \leq - \tau_1 \|d_k\| \leq - \tau_1 \tau_2\|P_{C}(x_k - \nabla f(x_k)) - x_k\|.
\]
Therefore, by  taking   limit in the last inequality as  $k\in K_1$ goes to infinity, we have
$$0 = \langle \nabla f(\tilde x) , s \rangle \leq - \tau_1 \tau_2\|P_{C}(\tilde x - \nabla f(\tilde x)) - \tilde x\|.$$
So, $\|P_{C}(\tilde x - \nabla f(\tilde x)) - \tilde x\| = 0$, which proves $\tilde x$ is a stationary point of \linebreak $\min_{x\in C} \{ f(x)=  \|F(x)\|^2/2\}.$
\end{proof}

\section{Numerical experiments}\label{numerical}

The purpose of these numerical experiments is to assess the practical behavior of G-LMM-IP.
For that, we consider two classes of nonlinear systems constrained to certain compact sets.
Firstly, we worked with box-constrained underdetermined systems and compared the performance
of G-LMM-IP with a well-known solver for bound-constrained least-squares problems.
Then, in the second set of test problems, we consider solving a system of equations over the spectrahedron,
where the use of inexact projections are essential to handle large-scale problems.

\subsection{Box-constrained systems}

This section reports some preliminary numerical experiments obtained by applying the G-LMM-IP
to solve 16 test problems of the form \eqref{eq:prob} with \linebreak $C =\{x \in \mathbb{R}^n: l \leq x \leq u\}$,
where $l, u \in\mathbb{R}^{n}$, see Table~\ref{table12}.
Most of them are small scale box-constrained underdetermined (or square) systems of nonlinear equations.
The last three, in fact, are defined by the set of nonlinear and bound constraints of optimization problems
from the CUTEr collection \cite{cuter}.

\begin{table}[h]
\centering
\caption{Test problems}\label{table12}
\vspace{0.2cm}
\begin{tabular}{llrr}
\hline\noalign{\smallskip}
Problem & Name and source & $m$ & $n$   \\
\noalign{\smallskip}\hline\noalign{\smallskip}
Pb 1  & Problem 46  from \cite{hock1981} & 2   & 5  \\
Pb 2  & Problem 53  from \cite{hock1981} & 3   & 5  \\
Pb 3  & Problem 56 from  \cite{hock1981} & 4   & 7 \\
Pb 4  & Problem 63  from \cite{hock1981} & 2   & 3 \\
Pb 5  & Problem 75  from \cite{hock1981} & 3   & 4 \\
Pb 6  & Problem 77  from \cite{hock1981} & 2   & 5  \\
Pb 7  & Problem 79  from \cite{hock1981} & 3   & 5 \\
Pb 8  & Problem 81  from \cite{hock1981} & 3   & 5 \\
Pb 9  & Problem 87  from \cite{hock1981} & 4   & 6  \\
Pb 10 & Problem 107 from \cite{hock1981} & 6   & 9  \\
Pb 11 & Problem 111 from \cite{hock1981} & 3   & 10 \\
Pb 12 & Problem 2 from \cite{Kanzow2004} & 150 & 300 \\
Pb 13 & Problem 4 from \cite{Kanzow2004} & 150 & 300 \\
Pb 14 & Problem EIGMAXA from \cite{cuter} & 101 & 101 \\
Pb 15 & Problem EIGMAXB from \cite{cuter} & 101 & 101 \\
Pb 16 & Problem EIGENA from \cite{cuter} & 2550 & 2550 \\
\hline
\end{tabular}
\end{table}

\begin{table}[h]
\centering
\caption{Performance of the G-LMM-IP and TRESNEI}\label{tab12}       
\vspace{0.2cm}
{\small
\begin{tabular}{cccccccccc}
\hline
 &  \multicolumn{3}{c}{G-LMM-IP ($M=1$)} &   \multicolumn{3}{c}{G-LMM-IP ($M=15$)} &    \multicolumn{3}{c}{TRESNEI}\\
Problem  & It & Fe & Time   & It & Fe & Time   & It & Fe & Time \\
\hline
Pb  1 & 8 & 9 & 0.19 & 8   & 9  & 0.20  & 8   & 9  & 0.18\\
Pb  2 & 1  &  2  & 0.03   & 1  & 2  & 0.03  & 2  & 3  & 0.05 \\
Pb  3  & 3  &  4  & 0.03 & 3   & 4  & 0.05 & 3   & 4  & 0.08  \\
Pb  4 & 5  &  8  & 0.05  & 5  & 8  & 0.06 & 5  & 6  & 0.07 \\
Pb  5 & 9  &  18  & 0.11  & 16   & 27  & 0.17  & 48   & 61  & 0.61 \\
Pb  6  & 6  &  7  & 0.05  & 6  & 7  & 0.06  & 6  & 7  & 0.08\\
Pb  7 & 4  &  5  & 0.05  & 4  & 5  & 0.05  & 4  & 5  & 0.06\\
Pb  8 & 8  &  9  & 0.09  & 8  & 9  & 0.10  & 109  & 114  & 1.11\\
Pb  9 & 48  &  49  & 0.60  & 48  & 49  & 0.62 & 54  & 55  & 0.67 \\
Pb  10 & 8  &  11  & 0.18   & 8  & 11  & 0.18  & 6  & 7  & 0.22\\
Pb  11 & 33  &  34  & 0.35   & 33  & 34  & 0.37  & 17  & 18  & 0.22 \\
Pb  12   &  2  &  3  & 0.05 & 2  & 3  & 0.05  & 18  & 19  & 0.64\\
Pb  13 & 9  &  10  & 0.20  & 9  & 10  & 0.20 & 16  & 17  & 0.42  \\
Pb 14 & 2 & 3 & 0.03 & 2 & 3 & 0.03 & 2 & 3 & 0.08 \\ 
Pb 15 & 12 & 13 & 0.13 & 10 & 11 & 0.08 & 20 & 29 & 0.42 \\ 
Pb 16 & 3 & 4 & 6.63 & 3 & 4 & 6.63 & 35 & 48 & 132.62 \\
\hline
\end{tabular}
}
\end{table}

We compare the performance of the G-LMM-IP with a Trust-Region Solver for Nonlinear Equalities and Inequalities (TRESNEI),
which is a MATLAB package based on the trust-region method \cite{porcelli12}, and available on  the web site  {\it http://tresnei.de.unifi.it/}.
The parameters of the TRESNEI were selected as recommended by the authors, see \cite[Subsection~6.2]{porcelli12}.
All numerical results were obtained using MATLAB R2018b on a 1.8GHz Intel\textregistered\ Core\texttrademark\ i5 with 8GB of RAM with MacOS 10.13.6 operating system.
The starting points and the bound constraints were defined as in \cite{hock1981,Kanzow2004}, 
except for the last three problems whose bounds and starting point are provided by the CUTEr package \cite{cuter}.
Moreover, we used the same overall termination condition  $\|F(x_k)\|\leq10^{-6}$. 
In the \linebreak G-LMM-IP, the initialization data were $M=1$, $M = 15$, $\eta_1 = 10^{-4}$, $\eta_2 = 10^{-2}$, $\eta_3 = 10^{10}$, $\gamma = 10^{-3}$, 
$\beta = 1/2$ and $\theta_k = 0$ for all $k$, (i.e., we consider exact orthogonal projection which is given explicitly by $P_{C}(x) = \min\{u , \max\{x,l\}\}$ in this application).
The linear systems in \eqref{A1:s21} were solved via QR factorization of the augmented matrix $( F'(x_k)^T  \ \sqrt{\mu_k}I )^T$.

Table~\ref{tab12} display the numerical results obtained for this test set. 
The methods were compared on the total number of iterations (It), number of $F$-evaluation (Fe) and CPU time in seconds (Time). 

From Table~\ref{tab12}, 
we can see that both G-LMM-IP and TRESNEI were able to solve all problems.
Regarding to the number of iterations, we observe that both versions of G-LMM-IP ($M=1$ and $M=15$)  are comparable to or even better than TRESNEI,
since they required less iterations in 8 out of 16 instances.
Similar behavior can also be observed for the number of F-evaluations.

Surprisingly, the monotone and the non-monotone versions of G-LMM-IP behaved quite similarly for this test set
(differences only occurred on problems 5 and 15). This happened because the full-step ($\alpha_k=1$) for the projected LM direction
was accepted in almost every iteration. Additionally, we remark that in almost all iterations the search direction 
was provided by Step~2 of G-LMM-IP, i.e., $d_k$ was the projected LM direction. Only for Problem 11 the projected gradient direction 
had to be used in 15 out of 33 iterations.


In summary, we can say that the G-LMM-IP seems to be reliable and competitive for solving small to medium-scale box-constrained systems of nonlinear equations.

\subsection{System of equations over the spectrahedron}

Let $F: \Sn \rightarrow \Rm$ be a continuous differentiable map from the set of symmetric matrices $\Sn$ to $\Rm$.
In this section, we consider the problem of finding a zero of $F$ belonging to the set of symmetric positive semidefinite
matrices with unit trace, denoted by
$$
\Snm = \{ X \in \Sn \ : \ \trace{X} = 1, \ X \succeq 0 \}.
$$
The set $\Snm$ is also known as spectrahedron.

Let $Y \in \Sn$ with spectral decomposition $Y = Q \Lambda Q^T$, where $Q$ is orthogonal and \linebreak
$\Lambda = \text{diag}(\lambda_1, \dots, \lambda_n)$
is a diagonal matrix with the eigenvalues of $Y$. It is a well-known fact that the projection
of $Y$ onto $\Snm$, with respect to the Frobenius norm (trace inner product), is given by
$$
Y_{+} = Q P_{\Delta_n} (\Lambda) Q^T,
$$
where $P_{\Delta_n} (\Lambda)$ corresponds to the projection of the eigenvalues of $Y$ onto the unit simplex in $\Rn$.
See \cite{ggrl2016}, and references therein, for further details.

Thus, the main burden for methods based on exact projections is the cost $O(n^3)$ of a full spectral decomposition
which turns prohibitive for general matrices of moderate size.

In this scenario, our concept of inexact projection becomes  interesting in practice.
As already mentioned in Remark~\ref{remark1d}, one way for computing an inexact projection of $Y_k$
onto $\Snm$ is by using the Conditional Gradient (Frank-Wolfe) method.

The $t$-iteration of the standard Frank-Wolfe method for solving
\begin{equation}\label{eq:projspec}
\begin{aligned}
\min & \quad (1/2) \| Z - Y_k \|_F^2 \\
\text{s.t} & \quad Z \in \Snm
\end{aligned}
\end{equation}
needs to find a $\bar{Z_t} \in \Sn$ solution of
\begin{equation}\label{eq:fwsub}
\begin{aligned}
\max_{Z} & \quad - \langle Z_t - Y_k, \, Z - Z_t \rangle \\
\text{s.t} & \quad \langle I, Z \rangle = 1, Z \succeq 0,
\end{aligned}
\end{equation}
where $Z_t$ is the current iterate.
%
The dual of (\ref{eq:fwsub}) is given by
\begin{equation}\label{eq:dual}
\begin{aligned}
\min_{\lambda} & \quad \lambda \\
\text{s.t} & \quad \lambda I - A \succeq 0,
\end{aligned}
\end{equation}
where $A = Y_k - Z_t$. Problem \eqref{eq:dual} is solved by determining the largest eigenvalue of $A$. Let $(\lambda, v)$ be such eigenpair (with $\|v\|=1$).
Then, the solution of (\ref{eq:fwsub})  is given by $\bar{Z}_t = vv^T$.

However, it is well-known that the classical version of Frank-Wolfe presents only $O(1/k)$ convergence rate \cite{jaggi13},
which means slow convergence, particularly when the tolerance $\varepsilon_k$ of the desired inexact projection
is relatively small.

For this reason, we consider
the method proposed in \cite{rankkfw}, which is a rank-$p$ variant of Frank-Wolfe that henceforth will be called FWp, for short.
Such method achieves linear convergence provided the solution of \eqref{eq:projspec} has rank $p$.
In fact, when specialized to problem \eqref{eq:projspec}, each iteration can be seen as
an ``inexact'' projected gradient iteration, where the next iterate $Z_{t+1}$ corresponds to
the matrix belonging to $\Snm$, with rank not greater than $p$, that is closest to $Z_t - \nabla \phi (Z_t)$,
where  $\phi(Z) = (1/2) \| Z - Y_k \|_F^2$ (i.e., the solution of \eqref{eq:projspec} with the additional constraint that $\text{rank}(Z) \le p$).
This subproblem demands the computation of the $p$ largest eigenvalues/vectors of $Z_t - \nabla \phi (Z_t)$.

Let $Y_k^+$ denotes the solution of \eqref{eq:projspec}. In case $\text{rank}(Y_k^+) \le p$, we retrieve the exact projection
in a single iteration of FWp. However, the overestimation of $\text{rank}(Y_k^+)$ may yield subproblems that are as expensive as
the full eigendecomposition of $Y_k$. In order to address this issue, we assume that an educate lower bound for $\text{rank}(Y_k^+)$
is available. Then, after each iteration of FWp, we check condition \eqref{eq:iproj} by solving \eqref{eq:fwsub} and if it is not verified,
we double the value of $p$. The last value of $p$ used in FWp for computing the approximate projection of $Y_k$ is stored to be
the initial guess for the rank of $Y_{k+1}^{+}$. For the computation of the $p$ largest eigenpairs we used the command {\tt eigs} from MATLAB.

The test problems where generated in the following way. We considered $F(X) = {\cal A}(X) - b$,
where ${\cal A}(X) = ( \langle A_1, X \rangle, \dots, \langle A_m, X \rangle )^T$.
We generate $X_* \in \Snm$ using the spectral decomposition $X_* = Q_* \Lambda_* Q_*^T$ with a random orthogonal matrix $Q$
and set $q$ eigenvalues in $\Lambda_*$ to $1/q$ and the remaining to zero.
Then, we build the matrices $A_{\ell} = (e_ie_j^T + e_je_i^T)/2$, for $\ell = 1,2, \dots, m$, where $e_i$ denotes a canonical vector of $\Rn$, and
the pairs $(i,j)$ correspond to the $m$ largest entries of $X_*$. Finally, $b_{\ell} = \langle A_{\ell}, X \rangle$, for $\ell = 1,2, \dots, m$.

Three different starting points were used: $X_0 = (1 - a) I/n + a \hat{X}$, where $\hat{X} = e_1 e_1^T$ and $a \in \{0, 1/2, 1\}$.
Thus, we vary from the geometric center of $\Snm$ to an extreme point of the feasible set.
The stopping criterion was set to $\| F(X_k) \| < 10^{-2}$ and the parameters of G-LMM-IP set to
$M = 1$, $\eta_1 = 10^{-2}$, $\eta_2 = 10^{-3}$, $\eta_3 = 10^{6}$, $\gamma = 10^{-3}$, $\beta = 1/2$, $\theta_k = 0.9$ for all $k$
in the inexact version and $\theta_k = 0$ in the exact one. In all instances, we have used the initial guess $p=1$ as for these instances
we known the rank of a solution $X_*$ was set to $q=4$.

In Table~\ref{tab:spec}, we present the number of iterations and CPU time in seconds for the class of problems discussed above,
varying the dimension $n$ and the number of equations $m$. 
We stress that the full projected LM step was accepted always in the line search.
As we can follow, although the number of iterations
may increase for some instances, 
the use of inexact projections provides considerable savings in terms of CPU time. 
In almost all  cases, the inexact version takes at least 50\% less CPU time  than the exact one, reaching, in some instances, 80\%.

These experiments indicate that G-LMM-IP-FWp as a promising alternative for solving system of equations over the spectrahedron
when it is expected that some solution in $C^*$ has low-rank.

Finally, in order to illustrate the superlinear versus linear local convergence rate, we solved again the instance with $n=1000$ and $m=200$, 
starting from $X_0 = I/n$, but using the refined stopping criterion $\| F(x_k) \| < 10^{-7}$. Table~\ref{tab:converg} shows the value of $\| F(x_k) \|$ in each iteration, 
which is a measure of $\dist(x_k, C^*)$ in view of the error bound condition {\bf (A2)}.

\begin{table}
\centering
\caption{Comparison of G-LMM with exact and inexact projections
for solving a system of equations over the spectrahedron.}\label{tab:spec}
\begin{tabular}{rrrcccc}
\hline
\ & \ &\  & \multicolumn{2}{c}{Exact} & \multicolumn{2}{c}{Inexact} \\ 
$n$ & $m$ & $\gamma$ & It & Time & It & Time \\ \hline
1000 & 200 & 0 & 2 & 0.93 & 4 & 0.74 \\
\ & & 0.5 & 15 & 4.75 & 15 & 1.85  \\
\ & & 1.0 & 19 & 5.23 & 19 & 2.24 \\ \hline
2000 & 400 & 0 &  2 & 5.11 & 4 & 3.28 \\
\ & & 0.5 & 15 & 29.49 & 15 & 7.50 \\
\ & & 1.0 & 19 & 36.53 & 19 & 9.18 \\ \hline
3000 & 600 & 0 & 2 & 15.33 & 4 & 8.20 \\
\ & & 0.5 & 15 & 91.70 & 15 & 16.93 \\
\ & & 1.0 & 19 & 115.69 & 19 & 20.34 \\ \hline
4000 & 800 & 0 & 2 & 34.82 & 4 & 15.55 \\
\ & & 0.5 & 15 & 206 & 15 & 34.64 \\
\ & & 1.0 & 19 & 258 & 19 & 37.24 \\ \hline
5000 & 1000 & 0 & 2 & 65 & 4 & 28.70 \\
\ & & 0.5 & 15 & 391 & 15 & 57.97 \\
\ & & 1.0 & 19 & 502 & 19 & 68.63 \\ \hline
\end{tabular}
\end{table}

\begin{table}
\centering
\caption{Superlinear ($\theta_k = 0$) versus linear ($\theta_k = 0.9$) convergence of LMM-IP.}\label{tab:converg}
\begin{tabular}{ccc}
\\ \hline 
\ & \multicolumn{2}{c}{$\| F(x_k) \|$} \\
$k$ & Exact & Inexact \\ \hline 
1 & 1.00E--01 & 6.17E--02 \\
2 & 1.40E--03 & 3.06E--02  \\ 
3 & 5.60E--06 & 1.50E--02 \\
4 & 2.24E--08 & 7.26E--03  \\
5 & & 3.37E--03  \\
6 & & 1.43E--03  \\
7 & & 4.56E--04  \\
8 & & 1.82E--06  \\
9 & & 7.34E--09  \\ \hline 
\end{tabular}
\end{table}

\section{Final remarks}\label{remarks}

This paper proposed and analyzed a Levenberg-Marquardt method  with  inexact projections for solving constrained  nonlinear systems.
For the local method, which combines the unconstrained Levenberg-Marquardt method with a type of the feasible  inexact projection, 
the local convergence as well as results on its rate were established under an error bound condition, which  is weaker than the standard full-rank condition of the $F'$. 
Then, a global version of this method has been proposed. 
It basically  consists of combining  our first algorithm, 
safeguarded by inexact projected gradient  steps,  with  a nonmonotone line search technique. 
The global convergence analysis was also presented.
Some numerical experiments were carried out in order to illustrate the numerical behavior of the methods. 
They indicate that the proposed schemes represent an useful tool for solving constrained  nonlinear systems mainly 
when  the orthogonal projection onto the feasible set can not be easily computed.


\section*{Acknowledgments}
\noindent The work of DSG was supported in part by CNPq Grant 421386/2016-9.
The work of MLNG and FRO was supported in part by CAPES, CNPq Grants 302666/2017-6, 408123/2018-4 
and FAPEG/CNPq/PRONEM-201710267000532.

\bibliographystyle{elsarticle-num}
\bibliography{RefInexactProjectLMN}

\end{document}